\documentclass[11pt,a4paper]{amsart}
\usepackage{amscd,amssymb,amsopn,amsmath,amsthm,mathrsfs,graphics,amsfonts,enumerate,verbatim,calc}
\usepackage[dvips]{graphicx}

\usepackage{amssymb,amsmath}
\usepackage{mathpazo}

\headsep=1cm \numberwithin{equation}{section}
\newtheorem{theorem}{Theorem}[section]
\newtheorem{main theorem}{Main Theorem}
\newtheorem{problem}{Problem}
\newtheorem{lemma}[theorem]{Lemma}
\newtheorem{proposition}[theorem]{Proposition}
\newtheorem{corollary}[theorem]{Corollary}

\theoremstyle{definition}
\newtheorem{definition}[theorem]{Definition}
\newtheorem{remark}[theorem]{Remark}

\newtheorem{example}[theorem]{Example}
\theoremstyle{remark}

\newtheorem{acknowledgement}{Acknowledgement}

\newcommand{\Ass}{\operatorname{Ass}}

\newcommand{\Spec}{\operatorname{Spec}}

\newcommand{\Div}{\operatorname{Div}}

\newcommand{\Ht}{\operatorname{ht}}

\newcommand{\Supp}{\operatorname{Supp}}

\newcommand{\depth}{\operatorname{depth}}

\newcommand{\Proj}{\operatorname{Proj}}
\newcommand{\sep}{\operatorname{sep}}

\newcommand{\fm}{\frak{m}}
\newcommand{\fp}{\frak{p}}
\newcommand{\fq}{\frak{q}}

\newcommand{\fn}{\frak{n}}

\begin{document}
\title[On the semicontinuity problem of fibers and global $F$-regularity]
{On the semicontinuity problem of fibers and global $F$-regularity}
\author[K.Shimomoto]{Kazuma Shimomoto}

\address{Department of Mathematics, College of Humanities and Sciences, Nihon University, Setagaya-ku, Tokyo 156-8550, Japan}
\email{shimomotokazuma@gmail.com}
\thanks{The author is partially supported by Grant-in-Aid for Young Scientists (B) \# 25800028}




\begin{abstract}
In this article, we discuss the semicontinuity problem of certain properties on fibers for a morphism of schemes. One aspect of this problem is local. Namely, we consider properties of schemes at the level of local rings, in which the main results are established by solving the lifting and localization problems for local rings. In particular, we obtain the localization theorems in the case of seminormal and $F$-rational rings, respectively. Another aspect of this problem is global, which is often related to the vanishing problem of certain higher direct image sheaves. As a test example, we consider the deformation of the global $F$-regularity.
\end{abstract}

\maketitle

\section {Introduction}

Let $f:X \to S$ be a morphism of locally Noetherian schemes and let $\mathbf{P}$ be a property on locally Noetherian schemes. Then we would like to investigate the following question. Let $U_f(\mathbf{P})$ denote the set of all $s \in S$ for which the base change scheme $X_s \times_{\Spec k(s)} \Spec L$ is $\mathbf{P}$, where $X_{s}:=f^{-1}(s)$ is the scheme-theoretic fiber and $k(s) \to L$ is any finite field extension. Then, is the set $U_f(\mathbf{P})$ open, closed, or constructible in the Zariski topology? In fact, this problem has been investigated in many interesting cases. More generally, let $\mathscr{F}$ be a coherent $\mathscr{O}_X$-module and let $\mathbf{P}$ be a property on coherent $\mathscr{O}_X$-modules. We set
$$
U_f^{\mathscr{F}}(\mathbf{P})=\{s \in S~|~\mathscr{F}_s:=i^{*}_{s}\mathscr{F}~\mbox{is}~\mathbf{P}~\mbox{for the inclusion map}~i_{s}:X_{s} \hookrightarrow X\}
$$ 
and then we can ask a similar question for $U_f^{\mathscr{F}}(\mathbf{P})$ as well. In these problems, the flatness condition is often crucial in order to apply results from commutative ring theory. It is interesting to know whether the main results in this paper are valid under the condition that a morphism is of finite flat dimension, since this case covers a local complete intersection morphism (see \cite{AvFo2} for related results). We also believe that they hold true for a pair $(X,D)$, where $D$ is a $\mathbb{Q}$-divisor of a variety $X$. The primary goal is to present a systematic treatment for the semicontinuity problem. 

Let $\psi:R \to S$ be a ring map of Noetherian rings and let $\mathbf{P}$ be a property on Noetherian rings. Then $\psi$ is said to be a $\mathbf{P}$-\textit{homomorphism}, if it is flat and $S \otimes_{R} L=S \otimes_R k(\fp) \otimes_{k(\fp)} L$ is $\mathbf{P}$ for any $\fp \in \Spec R$ and any finite field extension $k(\fp) \to L$ (see Definition \ref{P-map} below). A Noetherian ring $R$ is said to be a $\mathbf{P}$-\textit{ring}, if all of the formal fibers of $R$ are $\mathbf{P}$. The \textit{localization problem} is stated as follows \cite[(7.5)]{Gro2}.

\begin{problem}[Grothendieck]
Let $\psi:(R,\fm) \to (S,\fn)$ be a flat local map of Noetherian local rings. Assume that $R/\fm \to S \otimes_{R} R/\fm$ is a $\mathbf{P}$-homomorphism and $R$ is a $\mathbf{P}$-ring. Then is it true that $\psi$ is a $\mathbf{P}$-homomorphism?
\end{problem}

In fact, this problem has been investigated by many researchers extensively and positive answers have been obtained for the cases where, for example, $\mathbf{P}$=Cohen-Macaulay, Gorenstein, complete intersection, regular, reduced, normal, Serre's condition $(R_n)$, $(S_n)$, and so on. For a survey of these and other results, we refer the reader to \cite{AvFo}. The study of $\mathbf{P}$-rings and $\mathbf{P}$-homomorphisms started with a grand project on foundations of abstract algebraic geometry by Grothendieck in \cite{Gro2}. The \textit{lifting problem} is stated as follows.

\begin{problem}
For a Noetherian local ring $(R,\fm)$ together with a nonzero divisor $y \in \fm$, if $\mathbf{P}$ holds for $R/yR$, then does $\mathbf{P}$ lift to $R$?
\end{problem}

This problem has been established affirmatively in both trivial and nontrivial cases. For example, when $\mathbf{P}$=Cohen-Macaulay, this is trivial. Not obvious, but it is true that the lifting problem holds for $\mathbf{P}$=normal. Notably, the case of $\mathbf{P}$=seminormal was recently shown by Heitmann \cite{Hei}. 

Our main purpose in this article is to show that the properties of general fibers can be deduced from those of closed fibers for a morphism of schemes in both local and global cases. Most of our results in this article are derived from the following theorem (see Theorem \ref{Theorem1}).

\begin{main theorem}
Suppose that $f:X \to Y$ is a flat morphism of finite type of Noetherian excellent schemes and suppose the following conditions:

\begin{enumerate}
\item
$\mathbf{P}$ is defined at the level of excellent local rings;

\item
all excellent local rings have lifting property for $\mathbf{P}$;

\item
the image of every closed point of $X$ is closed in $Y$;

\item
$f_{k(s)}:X \times_Y \Spec k(s) \to \Spec k(s)$ is a $\mathbf{P}$-homomorphism for every closed point $s \in Y$.
\end{enumerate}

Then $f$ is a $\mathbf{P}$-homomorphism.
\end{main theorem}

The conditions $(1)$, $(2)$ and $(3)$ in the theorem are not so restrictive, while $(4)$ assures that $\mathbf{P}$ is stable under base change by a field extension on closed fibers. The condition $(3)$ is satisfied, when $f:X \to Y$ is a morphism between algebraic varieties defined over an algebraically closed field, or it is a proper morphism. The requirement of \textit{finite type} in the theorem is attributed to the use of Chevalley's theorem on constructible sets. The author believes that the \textit{finite type} condition is unnecessary in those cases considered in this article. We also note that a variant of the above theorem is proved by Grothendieck in \cite[Th\'eor\`eme 7.5.1]{Gro2}, in which the result is stated for a local map of complete local rings under the hypothesis that the residue field extension is finite. Marot \cite{Mar} also obtained some similar results in characteristic zero via resolution of singularities. Roughly speaking, the importance of the above theorem is expressed by the equivalence of the following conditions:
\begin{enumerate}
\item[$\bullet$]
$U_{f}(\mathbf{P})$ is a Zariski open set.

\item[$\bullet$]
$f$ is a $\mathbf{P}$-homomorphism and $U_{f}(\mathbf{P})$ is constructible.
\end{enumerate}

The set $U_{f}(\mathbf{P})$ is defined as previously. The structure of $U_{f}(\mathbf{P})$ was studied extensively by Grothendieck \cite{Gro2} under the assumption that the morphism is proper (see also \cite{GorWed} for a detailed list of this problem with citations from Grothendieck's EGA). In fact, most of the main results in this article hold for a \textit{going-up morphism} of schemes. This is a class of morphisms satisfying the going-up property for points of schemes. An important corollary is Corollary \ref{principle} which is established by regular alteration. It will be interesting to extend the main results of this article to the case of morphisms with finite flat dimension, using the methods developed in \cite{AvFo2}.

So far, we have considered only local conditions. As to global conditions, we will discuss the cases when $\mathbf{P}$=arithmetically Cohen-Macaulay, or globally $F$-regular. These notions are considered for projective varieties. In contrast to the local cases, we are not able to reduce the proof to the level of local rings. Nonetheless, the global property of a projective variety could be well understood by looking at the \textit{affine cone} with respect to a fixed embedding into a projective space. Among the main results, we mention the following result (see Corollary \ref{Finalcorollary}):

\begin{main theorem}[Criterion for global $F$-regularity]
Let $X$ be a connected normal projective variety over an $F$-finite field with $\dim X \ge 2$. Suppose that the following condition holds:
\begin{enumerate}
\item[$\bullet$]
There exists a generalized section ring $R=R(X,E)$ that is Gorenstein such that there is an injective $R$-module map $R \to R^{\frac{1}{q}};1 \mapsto c^{\frac{1}{q}}$, $R_c$ is strongly $F$-regular, where $c\in R$ is in no minimal prime of $R$, and its cokernel is a MCM module for some $q=p^e$.
\end{enumerate}
Then $X$ is globally $F$-regular.
\end{main theorem}

The notion of globally $F$-regular varieties was introduced and studied by Smith \cite{Smith}. The above result is obtained via use of local cohomology and canonical modules in the graded category. Throughout this article, an \textit{algebraic variety} over a field $k$ is always assumed to be a geometrically irreducible, reduced and separated scheme of finite type over $k$.

\section{Preliminaries}

Throughout this paper, we will assume that all rings and schemes are (locally) Noetherian. Let us recall the definition of $\mathbf{P}$-homomorphisms for rings and schemes, which was originally studied by Grothendieck \cite{Gro2}. Let $\fp \in \Spec R$ be a prime ideal. Denote by $k(\fp)$ the residue field of $R$ at $\fp$. For a ring homomorphism $\psi:R \to S$, we use the notation $\psi_{A}$ to denote $\psi \otimes_{R} A$ for an $R$-algebra $A$.

\begin{definition}
Let $\mathbf{P}$ be a ring theoretic property. 
\begin{enumerate}
\item
We say that  $\mathbf{P}$ is \textit{defined at the level of local rings}, if the following condition is satisfied: $\mathbf{P}$ holds for $R$ if and only if $\mathbf{P}$ holds for $R_{\fp}$ for all $\fp \in \Spec R$. 

\item
Let $X$ be a locally Noetherian scheme and assume that $\mathbf{P}$ is defined at the level of local rings. We say that $X$ is $\mathbf{P}$, if $\mathbf{P}$ holds for all of local rings $\mathscr{O}_{X,x}$.
\end{enumerate}
\end{definition}

Assume that $\mathbf{P}$ is defined at the level of local rings. Then it is true that a locally Noetherian scheme $X$ is $\mathbf{P}$ if and only if for every (or some) open covering $X=\bigcup_{\lambda \in \Lambda} U_{\lambda}$, where $U_{\lambda}=\Spec R_{\lambda}$, the ring $R_{\lambda}$ is $\mathbf{P}$.

\begin{definition}(\cite[(7.3.1)]{Gro2})
\label{P-map}
\begin{enumerate}
\item
Let $\psi:R \to S$ be a ring map of Noetherian rings. Then we say that $\psi$ is a $\mathbf{P}$-\textit{homomorphism}, if it is flat and if for any $\fp \in \Spec R$ and any finite field extension $k(\fp) \to L$, the Noetherian ring $S \otimes_R k(\fp) \otimes_{k(\fp)} L$ is $\mathbf{P}$.

\item
Let $f:X \to Y$ be a morphism of locally Noetherian schemes. Then we say that $f$ is a $\mathbf{P}$-\textit{homomorphism}, if $f$ is flat and if for any $x \in X$, there exist open affine subsets $U=\Spec S \subseteq X$ and $V=\Spec R \subseteq Y$ such that $x \in U$, $f(U) \subseteq V$ and the associated ring map $R \to S$ is a $\mathbf{P}$-homomorphism in the above sense.
\end{enumerate}
\end{definition}

We will mostly consider $\mathbf{P}$-homomorphisms when $\mathbf{P}$ is defined at the level of local rings in this article. If $\psi:R \to S$ is a $\mathbf{P}$-homomorphism and $R \to T$ is module-finite (not necessarily injective), the base change map $\psi _T:T \to S\otimes_R T$ is a $\mathbf{P}$-homomorphism, as the residue field extensions of $R \to T$ are finite.

Let $R$ be a Noetherian algebra over a field $K$. We say that a property $\mathbf{P}$ \textit{descends} under base change by a field extension of $K$, if $R \otimes_K L$ is $\mathbf{P}$ for some field extension $K \to L$, then $R$ is $\mathbf{P}$. Recall that a field extension $K \to L$ is a \textit{finitely generated separable extension}, if there is a finite set of elements $x_1,\ldots,x_d$ in $L$ such that $K(x_1,\ldots,x_d)$ is isomorphic to the field of fractions of the polynomial ring $K[x_1,\ldots,x_d]$, and $K(x_1,\ldots,x_d) \to L$ is finite separable in the usual sense. We refer the reader to \cite{Mat} for separable extensions of fields.

\begin{definition}
Let $\{K_{\lambda}\}_{\lambda \in \Lambda}$ be an inductive system consisting of a field $K$ and a family of field extensions $K \to K_{\lambda}$. We say that $\{K_{\lambda}\}_{\lambda \in \Lambda}$ is \textit{K-admissible}, if we are given $\lambda \le \lambda'$, then $K_{\lambda} \to K_{\lambda'}$ is a finite field extension, or a finitely generated separable extension.
\end{definition}

We note that $K_{\lambda} \to K_{\lambda'}$ in the above  definition can be a finite inseparable extension. The following lemma is useful when one checks the property $\mathbf{P}$ by taking only finite field extensions.

\begin{lemma}
\label{Lemma1}
Let $R$ be a finitely generated $K$-algebra for a field $K$. For any \textit{K-admissible} inductive system $\{K_{\lambda}\}_{\lambda \in \Lambda}$, assume that $\mathbf{P}$ is a property for which the following conditions hold:
\begin{enumerate}
\item[$\bullet$]
If $R \otimes_{K} K_{\lambda}$ is $\mathbf{P}$ for every $\lambda \in \Lambda$, then $R \otimes_{K} K'$ is $\mathbf{P}$ with $K'=\varinjlim K_{\lambda}$.

\item[$\bullet$]
Assume that $\mathbf{P}$ descends under base change by a field extension of $K$.
\end{enumerate}
Then the following conditions are equivalent:

\begin{enumerate}
\item
$R \otimes_{K} L$ is $\mathbf{P}$ for any finite field extension $L$ of $K$.

\item
$R \otimes_{K} L$ is $\mathbf{P}$ for any field extension $L$ of $K$.

\end{enumerate}
\end{lemma}

\begin{proof}
For any field extension $L$ of $K$, $R \otimes_{K} L$ is a finitely generated $L$-algebra. It suffices to show that $(1)$ implies $(2)$. Let $K'$ be any perfect field that is algebraic over $K$. Then let us first show that $R \otimes_{K} K'$ is $\mathbf{P}$. Since $K'=\varinjlim_{K \subseteq k \subseteq K'} k$, where $k$ runs over all finite subextensions of $K'$, it follows that the inductive limit $R \otimes_{K} K'=\varinjlim_{K \subseteq k \subseteq K'} R \otimes_{K} k$ is $\mathbf{P}$. In order to show that $R \otimes_{K} L$ is $\mathbf{P}$, taking the extension 
$$
R \otimes_{K} L \to (R \otimes_{K} L) \otimes_{L} K'L \simeq (R \otimes_{K} K') \otimes_{K'} K'L, 
$$
it suffices to show that $(R \otimes_{K} K') \otimes_{K'} K'L$ is $\mathbf{P}$ by the descent property of $\mathbf{P}$, and hence we may assume that $K$ is a perfect field by replacing $K'$ with $K$ in the above. Then by recalling that any finitely generated field extension of a perfect field is separable and any field extension of a perfect field is obtained as the inductive limit of such subextensions, we find that $R \otimes_{K} L$ is $\mathbf{P}$ by the inductive limit argument.
\end{proof}

As usual, let $\mathbf{P}$ be a property on Noetherian rings.

\begin{definition}
Let $(R,\fm)$ be a Noetherian local ring. Then we say that $R$ has \textit{lifting property for} $\mathbf{P}$, if $R/yR$ is $\mathbf{P}$ for a nonzero divisor $y \in \fm$, then so is $R$. We say that $R$ has \textit{specialization property for} $\mathbf{P}$, if $R$ is $\mathbf{P}$, then so is $R/yR$ for any nonzero divisor $y \in \fm$. 
\end{definition}

In practice, there are many cases where both lifting and specialization properties are known to hold for local rings. If one tries to consider these problems for non-local rings, counterexamples usually do exist.

\section{Local properties on fibers and some consequences}

Let $R$ be a reduced Noetherian ring such that the normalization map $R \to \overline{R}$ is finite (such a ring is called a  \textit{Mori ring}). For example, this property is satisfied by excellent rings, which naturally appear in many applications. We will see that it is essential to reduce the proof of the main theorem to a simple case by a standard technique, then we resort to topological arguments. First of all, we need the following lemma.

\begin{lemma}
\label{Lemma2}
Assume that $(R,\fm)$ is a Noetherian local domain with $\dim R \ge 2$. Then for any $f \in R$, the localization $R[f^{-1}]$ is not a field.
\end{lemma}

\begin{proof}
We are easily reduced to the case of $\dim R=2$, under which we keep the hypothesis. Then since the maximal ideal of $R$ is just the union of all height one primes, we find that there are infinitely many height one primes by the prime avoidance lemma. Then if $R[f^{-1}]$ is a field, it follows that the only prime ideal of $R$ which does not contain the principal ideal $(f)$ is $(0)$. But there are only finitely many height one primes containing the ideal $(f)$, we get a contradiction.
\end{proof}

We prove the first main theorem.

\begin{theorem}
\label{Theorem1}
Suppose that $f:X \to Y$ is a flat morphism of finite type of Noetherian excellent schemes and suppose the following conditions:

\begin{enumerate}
\item
$\mathbf{P}$ is defined at the level of excellent local rings;

\item
all excellent local rings have lifting property for $\mathbf{P}$;

\item
the image of every closed point of $X$ is closed in $Y$;

\item
$f_{k(s)}:X \times_Y \Spec k(s) \to \Spec k(s)$ is a $\mathbf{P}$-homomorphism for every closed point $s \in Y$.
\end{enumerate}

Then $f$ is a $\mathbf{P}$-homomorphism. 
\end{theorem}

\begin{proof}
Choose a point $\fp \in Y$. Let $V$ be the Zariski closure of $\fp$. Since $Y$ is a Noetherian topological space, $V$ contains a closed point $\fq \in Y$. Let $U=\Spec A \subseteq Y$ be an affine open subset such that $\fq \in U$. Then we have $\fp \in U$. So we may assume that both $X$ and $Y$ are affine by the condition $(1)$. Let $B$ be a finite type flat $A$-algebra such that $X=\Spec B$ and $Y=\Spec A$. Then we need to show that $A_{\fp} \to A_{\fp} \otimes_A B$ is a $\mathbf{P}$-homomorphism. Now let $R=A_{\fq}$ and $S=A_{\fq} \otimes_A B$. Then we are reduced to showing that the flat map $\psi:R \to S$ is a $\mathbf{P}$-homomorphism, where $S$ is a finite type flat $R$-algebra. Assume that this is false. Then since the closed fiber of $\psi$ is $\mathbf{P}$, using Noetherian induction, we may choose $\fp \in \Spec R$ such that $k(\fp') \to S \otimes_{R} k(\fp')$ is a $\mathbf{P}$-homomorphism for every $\fp' \in \Spec R$ with $\fp \subsetneq \fp'$, while $k(\fp) \to k(\fp) \otimes_{R} S$ is not. Let us replace the original $\psi$ with $R/\fp \to S/\fp S$.

Let $R'$ be the integral closure of $R$ in a finite field extension of the quotient field of $R$. Then the induced map $\psi_{R'}: R' \to S \otimes_{R} R'$ fulfills the condition (4), because $R \to R'$ is module-finite. Now replacing $R$ with $R'$, we may assume that $R$ is a semilocal normal domain. 

Under the assumptions as above, it is clear that it is sufficient to show that the generic fiber of $\psi:R \to S$ is $\mathbf{P}$. Since $R$ is an excellent domain, the regular locus $Y_{\mathrm{reg}} \subseteq \Spec R$ is a non-empty Zariski open subset. We claim that the local ring $S_{P}$ is $\mathbf{P}$ for $P \in \Spec S$ such that $\fp:=R \cap P$ is nonzero and $\fp \in Y_{\mathrm{reg}}$. In fact, since the maximal ideal $\fp$ of $R_{\fp}$ is generated by a regular sequence, and since $S_{P}/\fp S_{P}$ is $\mathbf{P}$ by hypothesis, it follows that $S_{P}$ is $\mathbf{P}$ by the condition $(2)$.

Now pick any $P \in \Spec S$ with $R \cap P=(0)$. To prove the theorem, it suffices to find a prime $Q \in \Spec S$ such that 
\begin{equation}
\label{crucialpoint}
R \cap Q \in Y_{\mathrm{reg}},~P \subseteq Q~\mbox{and}~R \cap Q \ne (0).
\end{equation}

Let $V$ denote the Zariski closure of the point $P \in \Spec S$. We claim that the image $Z:=f(V)$ for the map $f:\Spec S \to \Spec R$ contains a Zariski open subset. Indeed, Chevalley's theorem asserts that $Z$ is constructible and contains the generic point $(0) \in \Spec R$. Hence $Z$ contains a dense open subset since $\Spec R$ is integral. We divide the rest of the proof into two cases. 

Assume first $\dim R=1$, in which case $R$ is a semilocal Dedekind domain. Pick any $P \in \Spec S$ such that $R \cap P=(0)$. Since the pull-back of every maximal ideal of $S$ is maximal in $R$ by the condition $(3)$, there exists $Q \in \Spec S$ such that $P \subseteq Q$ and $\fq=R \cap Q$ is a maximal ideal of $R$. Applying the condition $(2)$ again to the induced map $\psi_Q:R_{\fq} \to S_Q$, it follows that the localization  $S_P$ of $S_Q$ is $\mathbf{P}$. 
   
Assume next $\dim R \ge 2$. The intersection $Z':=Y_{\mathrm{reg}} \cap Z$ is a constructible subset of $Y$ and $(0) \in Z'$, where $(0)$ is the generic point of $Y$. This implies that $Z'$ contains a dense open subset of $Y$. By shrinking $Z'$, we may assume that $Z'$ is an open subset of $Y$ and we claim that $(0) \subsetneq Z'$. Indeed, if $(0)=Z'$, the generic point $(0)$ must be an open subset of $Y$. This implies that there exists a nonzero element $f \in R$ such that the localization $R[f^{-1}]$ is a field. However, this is impossible, due to Lemma \ref{Lemma2}. 

By these observations, we can find a prime ideal $Q \in \Spec S$ satisfying $(\ref{crucialpoint})$. Thus, we have shown the theorem in all cases.
\end{proof}

\begin{remark}
This theorem is useful when one wants to deduce certain properties of general fibers from those of closed fibers. We state some corollaries as immediate consequences of the main theorem. For the definition of seminormal rings and their basic properties, we follow \cite{GrTr}, where one finds interesting applications to singularities on schemes and analytic varieties.
\end{remark}

\begin{corollary}[Localization theorem I]
Let $\phi:R \to S$ be a flat map of finite type of excellent rings. Assume that $P \cap R$ is a maximal ideal for all maximal ideals $P$ of $S$ and every closed fiber of $\phi$ is geometrically seminormal over the corresponding residue field of $R$. Then $\phi$ is a seminormal homomorphism.
\end{corollary}

\begin{proof}
The lifting property is due to Heitmann \cite{Hei} and seminormality is a local property. The corollary follows from Theorem \ref{Theorem1}.
\end{proof}

\begin{definition}
We say that a local ring $R$ (or a scheme $X$) is \textit{standard}, if $R$ (or $X$) is essentially of finite type over a field. 
\end{definition}

All standard schemes are excellent. We use this terminology only in this article.

\begin{corollary}
Let $f:X \to Y$ and $g:Z \to Y$ be morphisms of algebraic varieties over an algebraically closed field $k$ such that $f:X \to Y$ is flat, every closed fiber of $f$ is $\mathbf{P}$, and $Z$ is $\mathbf{P}$. Assume that:

\begin{enumerate}
\item
all standard local rings have lifting property for $\mathbf{P}$;
 
\item
$\mathbf{P}$ is defined at the level of standard local rings;

\item
$\mathbf{P}$ has the property that if $(R,\fm) \to (S,\fn)$ is a flat local map of standard local rings such that $R$ and all fibers are $\mathbf{P}$, then so is $S$.
\end{enumerate}

Then $X \times_{Y} Z$ is $\mathbf{P}$.
\end{corollary}

\begin{proof}
Since $k$ is an algebraically closed field, $X \times_{Y} Z$ is a $k$-variety and the base change morphism $f \times \mathrm{id}_{Z}:X \times_{Y} Z \to Z$ is flat, which maps closed points to closed points by Hilbert's Nullstellensatz. The fiber of $X \times_{Y} Z \to Z$ over a closed point $z \in Z$ is $X \times_{Y} Z \times_{Z} \Spec k(z) \simeq X \times_{Y} \Spec k$. So it follows from Theorem \ref{Theorem1} that $X \times_{Y} Z \to Z$ is a $\mathbf{P}$-homomorphism. Replacing $X \times_{Y} Z \to Z$ with the induced local map of local rings of schemes, we use the condition $(3)$ to conclude that $X \times_{Y} Z$ is $\mathbf{P}$. 
\end{proof}

\begin{remark}
\label{rmk1}
Let $\mathbf{P}$ be one of the following properties: Cohen-Macaulay, Gorenstein, locally complete intersection, regular, normal, or seminormal. Then the above corollary holds for these cases. The lifting property in the above cases, except for the case $\mathbf{P}$=seminormal, which is due to Heitmann, is almost immediate from definitions. For the seminormality for rings and its geometric side, we refer the reader to \cite{GrTr} and \cite{Vit}. Note that seminormality also appears in the compactification problem of moduli spaces of higher dimensional varieties \cite{Kol}.
\end{remark}

For a scheme map $f:X \to Y$, we use the following notation. We denote by $U_{f}(\mathbf{P})$ the set of all scheme-theoretic points of $Y$ such that $y \in U_{f}(\mathbf{P})$ $\iff$ $X_y \times_{\Spec k(y)} \Spec L$ is $\mathbf{P}$ for $X_y:=f^{-1}(y)$ and any finite field extension $k(y) \to L$.

\begin{corollary}[Semicontinuity principle]
\label{principle}
Suppose that $f:X \to Y$ is a proper flat morphism of integral schemes of finite type over a field $k$ and suppose the following conditions:

\begin{enumerate}
\item
all standard local rings have both lifting and specialization properties for $\mathbf{P}$;

\item
$\mathbf{P}$ is defined at the level of standard local rings;

\item
the $\mathbf{P}$-locus in every standard scheme is constructible;

\item
if $R$ is a standard local ring over a field $K$, then $R$ is $\mathbf{P}$ $\iff$ $R \otimes_{K} L$ is $\mathbf{P}$ for every finitely generated field extension $K \to L$.
\end{enumerate}

Then $U_{f}(\mathbf{P})$ is Zariski open in $Y$.
\end{corollary}

\begin{proof}
First, we prove the corollary for the case when $Y$ is a regular scheme. What we need to show is that, if $y \in Y$ satisfies the property that $X_{y}=X \times_{Y} \Spec k(y)$ is $\mathbf{P}$, then there exists an open neighborhood $y \in U$ such that every fiber over $U$ is $\mathbf{P}$ and use the condition $(4)$. 

Recall that a subset of a Noetherian scheme is open if and only if it is constructible and stable under generalization. The base change map $X \times_Y \Spec \mathscr{O}_{Y,y} \to \Spec \mathscr{O}_{Y,y}$ sends closed points to closed points, since it is proper. Thanks to $(1)$ and $(2)$, $X \times_Y \Spec \mathscr{O}_{Y,y} \to \Spec \mathscr{O}_{Y,y}$ is a $\mathbf{P}$-homomorphism in view of Theorem \ref{Theorem1}. Then it suffices to show that $U_{f}(\mathbf{P})$ is constructible. Let $f(x)=y$ for $y \in Y$ as above. Then the closed fiber of the local map $\mathscr{O}_{Y,y} \to \mathscr{O}_{X,x}$ is $\mathbf{P}$ and $\mathscr{O}_{Y,y}$ is regular by hypothesis, so we can lift $\mathbf{P}$ from the fiber $\mathscr{O}_{X,x}/\fm_{\mathscr{O}_{Y,y}}\mathscr{O}_{X,x}$ to $\mathscr{O}_{X,x}$ by the lifting property of $\mathbf{P}$. Let $V \subseteq X$ be the maximal $\mathbf{P}$-locus. Then $V$ contains the fiber $X_{y}$. By using $(3)$, $V$ is constructible. Let $W$ denote the complement of $V$ in $X$. Let $U$ be the complement of $f(W)$ in $Y$. We find that $U$ is the maximal subset of $Y$ such that the fiber over every point of $U$ is $\mathbf{P}$, as $\mathbf{P}$ has both lifting and specialization properties. By Chevalley's theorem, it follows that $U$ is constructible. Hence $U$ is non-empty and open in the Zariski topology.

Next assume that $Y$ is any integral scheme which is of finite type over $k$. Then by the existence of regular alterations \cite{deJo}, there exist a regular integral scheme $Y'$ together with a generically finite, proper map $g:Y' \to Y$ and the fiber product diagram:
$$
\begin{CD} 
X \times_{Y} Y' @>>> Y' \\
@VVV @VVgV \\
X @>f>> Y \\
\end{CD}
$$
To prove the general case, let $y' \in Y'$ and let $y=g(y')$. Then we have $(X \times_{Y} Y') \times_{Y'} \Spec k(y') \simeq X \times_{Y} \Spec k(y')$, where $k(y) \to k(y')$ is a finitely generated extension. Since the fiber $(X \times_{Y} Y')_{y'}$ is the base change of $X_y$ by a finitely generated field extension, the condition (4) implies that $X_y$ is $\mathbf{P} \iff (X \times_{Y} Y')_{y'}$ is $\mathbf{P}$. As $Y'$ is regular, we can find an open subset $U' \subseteq Y'$ as previously, and $U:=g(U')$ is constructible. Then apply Theorem~\ref{Theorem1} to conclude that $U$ is open, as desired.
\end{proof}

\begin{remark}
\label{rmk2}
All assumptions in the above corollary are satisfied more specifically in the case of Cohen-Macaulay, locally complete intersection, and Gorenstein fibers. In fact, the corollary holds for an arbitrary proper flat morphism of excellent integral schemes for these cases at least in characteristic zero by the existence of desingularizations of quasi-excellent schemes of characteristic zero proved by Temkin \cite{Tem}.
\end{remark}

In characteristic $p>0$, there are distinguished classes of Noetherian rings defined via the Frobenius map, which are studied in tight closure theory. It should be pointed out that, as was shown by Fedder \cite{Fed} and Singh \cite{Singh}, the lifting property fails for $F$-regular and $F$-pure rings. One also cannot expect that tight closure commutes with localization (see \cite{BrMo} for a counterexample to this problem). We refer the reader to \cite{Hu} for an account of tight closure theory. 

Here, we state the localization theorem for $F$-rational rings, which admit lifting property, and we do not require any essential part of tight closure theory for the proof. So we only recall the definition. Let $I$ be an ideal of a Noetherian ring $R$ of characteristic $p>0$. Then the \textit{tight closure}, denoted by $I^*$, is the set of all $x \in R$ such that $cx^{p^e} \in I^{[p^e]}$ for $e \gg 0$ and some $c \in R$ not in any minimal prime of $R$. Then the tight closure $I^*$ is an ideal containing $I$.

A (not necessarily local) Noetherian ring $R$ of characteristic $p>0$ is \textit{F-rational}, if every parameter ideal $I$ of $R$ is tightly closed. An ideal $I$ is a \textit{parameter ideal} if $\Ht(I)=\mu(I)$, where $\mu(I)$ is the minimal number of generators of $I$. In fact, an excellent ring $R$ is $F$-rational if and only if the localization $R_{\fp}$ is $F$-rational for all prime ideals $\fp$ of $R$. The following corollary partially answers a question of Hashimoto \cite[Remark 6.7]{Hashi1}.

\begin{corollary}[Localization theorem II]
Let $\phi:R \to S$ be a flat map of finite type of excellent rings of characteristic $p>0$. Assume that $P \cap R$ is a maximal ideal for all maximal ideals $P$ of $S$ and every closed fiber of $\phi$ is geometrically $F$-rational over the corresponding residue field of $R$. Then $\phi$ is an $F$-rational homomorphism.
\end{corollary}

\begin{proof}
Let $A$ be an excellent local ring. Then it is $F$-rational if and only if the $\fm$-adic completion $\widehat{A}$ is so. Note that $\widehat{A}$ is a residue class ring of a Cohen-Macaulay ring. Then we use this fact together with the lifting property for $F$-rational local rings \cite[Theorem 4.2]{HoHu2}, so that $F$-rational excellent rings fulfill the hypotheses of Theorem \ref{Theorem1}.  
\end{proof}

Finally, we prove the generic smoothness of the following form.

\begin{proposition}[Generic smoothness]
\label{Smooth}
Let $f:X \to S$ be a proper flat morphism of Noetherian schemes. Assume that the fiber $X_{s}=X \times_{S} \Spec k(s)$ is smooth over $k(s)$. Then there exists an open subset $U \subseteq S$ with $s \in U$ such that $f:f^{-1}(U) \to U$ is smooth.
\end{proposition}

\begin{proof}
To simplify notation, we use the same symbol $x$ for a point in $X_s$ or $X$. Note that the sheaf of relative K\"{a}hler differentials $\Omega_{X/S}$ is coherent. Let $\dim_x f$ denote the Krull dimension of the fiber $X_s$ at $x \in X$ with $f(x)=s$.

According to \cite[Corollaire 17.5.2 and Proposition 17.15.15]{Gro2}, our assumptions imply that there exists a maximal open subset $Z \subseteq X$ such that $\Omega_{X/S}$ is locally free over $Z$ of rank equal to $\dim_x f$ at every point $x \in Z$, with $X_{s} \subseteq Z$. Let $Z^{\mathrm{c}}$ be the complement of $Z$ in $X$. Since $f$ is proper, $f(Z^{\mathrm{c}})$ is closed and $s \notin f(Z^{\mathrm{c}})$. Let $U \subseteq S$ be the complement of $f(Z^{\mathrm{c}})$. Then $U$ is open in $S$ and $f:f^{-1}(U) \to U$ is smooth, as required.
\end{proof}

\section{Semicontinuity of global properties on fibers}

We consider the deformation problem for global properties on fibers for a morphism of schemes. The global cases are often related to the vanishing of certain kind of sheaf cohomology modules. In particular, it is essential to relate the sheaf cohomology modules to the local cohomology modules.

We begin with notation for sheaf and local cohomology modules for graded modules (\cite{Har} and \cite{GoWa} as standard references), and Weil divisors with $\mathbb{Q}$-coefficients on projective varieties, which will be used throughout. 

We will assume that the graded ring $R=\bigoplus_{n \ge 0} R_{n}$ is Noetherian with its irrelevant graded ideal $\fm=\bigoplus_{n > 0}R_{n}$. For a graded $R$-module $M$, we put $M(n):=M \otimes_{R} R(n)$ for $n \in \mathbb{Z}$, where $[R(n)]_i:=R_{n+i}$. We often consider the case where $R_0=K$ is a field and thus, $\fm=\bigoplus_{n > 0} R_{n}$ is the unique graded maximal ideal.  For a graded $R$-module $M$, we denote by $H_{\fm}^{i}(M)$ the local cohomology module with support at $\fm$. So this is a graded $R$-module. The \textit{canonical module} $K_R$ is defined as the Matlis dual of $H^{\dim R}_{\fm}(R)$.

Assume that $X$ is a connected normal projective variety over a field $k$. Let $\Div(X)$ be the group of Weil divisors on $X$ and let $\Div(X)_{\mathbb{Q}}:=\Div(X) \otimes \mathbb{Q}$. Let $(f)$ be the divisor associated to a nonzero rational function $f \in k(X)$, where $k(X)$ is the function field of $X$. We say that a $\mathbb{Q}$-divisor $E \in \Div(X)_{\mathbb{Q}}$ is \textit{ample}, if some multiple $nE$ is an ample Cartier divisor for some $n>0$. For a Weil divisor $D$, the sheaf $\mathscr{O}_X(D)$ is defined as a subsheaf of the constant sheaf $k(X)$ by assigning to each open subset $U \subseteq X$ the set of $\{f \in k(X)~|~f=0~\mbox{or}~(f)+D \ge 0~\mbox{on}~U\}$. For $E=\sum r_V \cdot V \in \Div(X)_{\mathbb{Q}}$, we denote by $\lfloor E \rfloor$ the integral divisor $\sum \lbrack r_V \rbrack \cdot V$, where $\lbrack r_V \rbrack$ is the largest integer not exceeding $r_V$. Then let $\mathscr{O}_X(nD):=\mathscr{O}_X(\lfloor nD \rfloor)$. Finally, let $\mathscr{M}(nD):=\mathscr{M} \otimes \mathscr{O}(nD)$.

\begin{definition}[Generalized section ring]
Let $D$ be an ample $\mathbb{Q}$-divisor on a normal projective variety $X$. Then the $\textit{generalized section ring}$ of $X$ with respect to $D$ is defined as 
$$
R=R(X,D):=\bigoplus_{n \in \mathbb{Z}_{\ge 0}} H^{0}(X,\mathscr{O}_{X}(nD)),
$$ 
where we put $H^0(X,\mathscr{O}_X(nD)):=\{f \in k(X)~|~(f)+nD \ge 0\} \cup \{0\}.$
\end{definition}

When $D$ is an ample Cartier divisor, we simply say that $R(X,D)$ is a \textit{section ring}. Then $H^{0}(X,\mathscr{O}_{X}(-nD))=0$ for $n>0$ and an ample Cartier divisor $D$ on an integral projective scheme by \cite[Chapter III, Exercise 7.1]{Har}, so we can take $\mathbb{Z}$ instead of $\mathbb{Z}_{\ge 0}$ as an index set in the setting of section rings. If there is no danger of confusion, we simply write the (generalized) section ring as $R$. The generalized section ring $R(X,D)$ is a finitely generated algebra over a field $H^0(X,\mathscr{O}_X)$ and it recovers $X$, namely, an isomorphism $X \simeq \Proj R(X,D)$ when $X$ is a connected normal variety. Let $K=H^0(X,\mathscr{O}_X)$. Then $k \subseteq K$ is a finite extension of fields. 

If $\mathscr{M}$ is a coherent $\mathscr{O}_X$-module, then 
$$
M:=\bigoplus_{n \in \mathbb{Z}} H^{0}(X,\mathscr{M} \otimes \mathscr{O}_{X}(nD))
$$ 
is a graded $R$-module, and we have $\widetilde{M} \simeq \mathscr{M}$, where $\widetilde{M}$ denotes the associated $\mathscr{O}_X$-module. The verification of this fact is easily reduced to the case of very ample Cartier divisors, where the assertion is well-known, by taking the "Veronese submodule" of $M$. $M$ is finitely generated over $R$ under a very mild condition (see Lemma \ref{Lemma3} below).

Let $M$ be a finitely generated graded $R=R(X,D)$-module. There is an isomorphism of graded $R$-modules:
$$
\bigoplus_{n \in \mathbb{Z}} H^{i}(X,\widetilde{M} \otimes \mathscr{O}_X(nD)) \simeq H^{i+1}_{\fm}(M)
$$
for all $i \ge 1$ (see \cite[Corollaire 2.1.4]{Gro1} for Cartier divisors and \cite[Proposition 2.2]{Wa} for $\mathbb{Q}$-divisors together with an identification $\widetilde{R(n)} \cong \mathscr{O}_X(nD)$ in \cite[Lemma 2.1]{Wa}). Let $f_{1},\ldots,f_{d}$ be a homogeneous system of parameters of $R$ so that $\fm=\sqrt{(f_{1},\ldots,f_{d})}$. For a graded $R$-module $M$, the local cohomology $H^{i}_{\fm}(M)$ can be computed from the modified \v{C}ech complex:
$$
0 \to M \to \bigoplus_{f_i} M_{f_i} \to \bigoplus_{f_i < f_j} M_{f_i f_j} \to \cdots.
$$
Taking the graded components, we have $[H^{i+1}_{\fm}(M)]_{n} \simeq H^{i}(X,\widetilde{M} \otimes\mathscr{O}_X(nD))$ for $i \ge 1$.

\begin{lemma}
\label{Lemma3}
Assume that $X$ is a connected normal projective variety and $\mathscr{M}$ is a coherent $\mathcal{O}_X$-module without associated components of dimension $0$. Set
$$
M=\bigoplus_{n \in \mathbb{Z}} H^{0}(X,\mathscr{M} \otimes \mathscr{O}_X(nD))
$$ 
for an ample $\mathbb{Q}$-divisor $D$. Then $M$ is a finitely generated graded module over $R=R(X,D)$ and $\depth M \ge 2$.
\end{lemma}

\begin{proof}
Recall that there is an identification $\widetilde{R(n)} \cong \mathscr{O}_X(nD)$ \cite[Lemma 2.1]{Wa}. Under our assumption, $M$ is finitely generated \cite[Notation 5.1.4]{GoWa}. A part of the modified \v{C}ech complex just yields the following exact sequence \cite[(2.1.5.)]{Gro1}
$$
0 \to H^{0}_{\fm}(M) \to M \to \bigoplus_{n \in \mathbb{Z}} H^{0}(X,\widetilde{M} \otimes \mathscr{O}_X(nD)) \to H^{1}_{\fm}(M) \to 0.
$$
The middle map is an isomorphism by the definition of $M$. 
From this, it follows that $\depth M \ge 2$, as desired.
\end{proof}

The global properties defined on projective varieties are often related to the vanishing property on certain cohomology modules. We show this first for \textit{arithmetically Cohen-Macaulay} ("ACM" for short) varieties.

\begin{definition}
Let $X$ be a connected normal projective variety over a field $k$. 
\begin{enumerate}

\item
We say that $X$ is \textit{ACM} (with respect to $D$), if the section ring $R(X,D)$ is Cohen-Macaulay for some ample $\mathbb{Q}$-divisor $D$. 

\item
Let $\mathscr{M}$ be a coherent sheaf on $X$. We say that $\mathscr{M}$ is \textit{ACM} (with respect to $D$), if $\mathscr{M}$ is a CM sheaf and if there exists an ample $\mathbb{Q}$-divisor $D$ such that $H^{i}_{\fm}(M)=0$ for all $0 \le i \le \dim X$, where $M:=\bigoplus_{n \in \mathbb{Z}} H^{0}(X,\mathscr{M} \otimes \mathscr{O}_X(nD))$ as a graded module over $R(X,D)$ with a graded maximal ideal $\fm$.
\end{enumerate}
\end{definition}

It is important to keep in mind that our definition is different from the usual definition of ACM varieties. The usual definition says that $X$ is ACM, if the homogeneous coordinate ring of $X$ with respect to some embedding into a projective space is CM. As the current version is flexible, we take it as our definition.

\begin{remark}
\label{rmk3}
\begin{enumerate}
\item
The CM condition is local, while the ACM condition is global. However, if $X$ is ACM, it is locally CM; every local ring $\mathscr{O}_{X,x}$ is CM \cite[Remark (2.11)]{Wa}.

\item
Note that the ACM condition depends on the divisor $D$. When necessary, we shall say that "$\mathscr{M}$ is ACM with respect to $D$". We also use the phrase that "$\mathscr{M}$ is ACM with respect to every ample $\mathbb{Q}$-divisor", when the ACM condition is satisfied for every such divisor.

\item
By taking Veronese subrings (or submodules), it is shown that a sheaf $\mathscr{M}$ is ACM with respect to some ample $\mathbb{Q}$-divisor if and only if $\mathscr{M}$ is ACM with respect to some ample Cartier divisor.

\item
If $\Ass_{\mathscr{O}_X}(\mathscr{M})$ has a component of dimension $0$, then $M$ is, in general, not a finitely generated $R=R(X,D)$-module. For example, take $X=\mathbb{P}^1_k$ and take $\mathscr{M}$ to be the structure sheaf of a closed point of $\mathbb{P}^1_k$. Then $\Gamma(\mathbb{P}^1_k,\mathscr{M}(n))=k$ for all $n \in \mathbb{Z}$. Hence $M$ is not finitely generated over $k[x]=\bigoplus_{n \ge 0} H^0(\mathbb{P}^1_k,\mathscr{O}(n))$.

\item
Let $\overline{k}$ be an algebraic closure of $k$. The ACM condition holds over $\mathscr{M}$ on $X$ if and only if the ACM condition holds over $\mathscr{M} \otimes \overline{k}$ on $X \times \Spec \overline{k}$.
\end{enumerate}
\end{remark}

\begin{definition}
Let $R=\bigoplus_{n \ge 0} R_n$ be a Noetherian graded ring over a field $R_0=k$ with $\fm=\bigoplus_{n>0} R_n$. Let $M$ be a nonzero finitely generated graded $R$-module. Then we say that $R$ is a \textit{graded MCM module}, if $H^i_{\fm}(M)=0$ for all $i<\dim R$. Equivalently, every homogeneous system of parameters of $R$ is a regular sequence on $M$.
\end{definition}

By \cite[(2.1.27)]{BrHer}, a graded $R$-module $M$ is graded MCM if and only if $M_{\fm}$ is MCM over $R_{\fm}$. In particular, a graded MCM module $M$ is a CM module over $R$ and we use this fact in the proof of the next proposition.

\begin{proposition}(\cite[Proposition 2.1]{CaHar})
\label{Proposition1}
Let $X$ be a connected normal projective variety and fix an ample $\mathbb{Q}$-divisor $D$ on $X$. Then there is a one-to-one correspondence between ACM sheaves with respect to $D$ and graded MCM modules over $R=R(X,D)$. This correspondence is given by 
$$
\mathscr{M} \mapsto M=\bigoplus_{n \in \mathbb{Z}} H^{0}(X,\mathscr{M}\otimes \mathscr{O}_X(nD)).
$$
Hence $\mathscr{M}$ is ACM with respect to $D$ if and only if $M$ is a graded MCM module over $R$.
\end{proposition}

\begin{proof}
It suffices to prove that if $M$ is a graded MCM module, then $\widetilde{M}$ is a CM sheaf. Note that there exists an integer $n_0>0$ such that the $n$-th Veronese subring $R^{(n)}$ of $R$ is a standard graded ring over the field $k=H^0(X,\mathscr{O}_X)$ for every $n \ge n_0$. Let $\fp$ be a graded prime ideal of $R$ not containing the irrelevant maximal ideal of $R$, and let $M_{(\fp)}$ be the homogeneous localization. Then $M_{(\fp)}$ is a CM module over $R_{(\fp)}$. Since $R_{(\fp)}$ has a unit element of degree one, we have $R_{(\fp)}=(R_{(\fp)})_0[T,T^{-1}]$ and so $M_{(\fp)}=(M_{(\fp)})_0[T,T^{-1}]$ (see \cite[Sublemma (5.1.3) and Lemma (5.1.10)]{GoWa}). From this description, $(M_{(\fp)})_0$ is a CM module over $(R_{(\fp)})_0$. Thus, the sheaf $\widetilde{M}$ is CM, as desired.
\end{proof}

We prove the following characterization of ACM sheaves.

\begin{lemma}
\label{Lemma4}
Let $X$ be a connected normal projective variety and let $\mathscr{M}$ be a coherent sheaf on $X$ without associated components of dimension $0$. Then the following conditions are equivalent:

\begin{enumerate}
\item
The $\mathscr{O}_X$-module $\mathscr{M}$ is an ACM sheaf.

\item
The $\mathscr{O}_X$-module $\mathscr{M}$ is CM and $H^i(X,\mathscr{M}\otimes \mathscr{O}_X(nD))=0$ for some ample $\mathbb{Q}$-divisor $D$, all $0<i<\dim X$, and all $n \in \mathbb{Z}$.
\end{enumerate}
\end{lemma}

\begin{proof}
We put $M=\bigoplus_{n \in \mathbb{Z}} H^{0}(X,\mathscr{M}\otimes \mathscr{O}_X(nD))$. Then $M$ is a graded module over $R(X,D)$. Under the notation as above, it follows that $\mathscr{M} \simeq \widetilde{M}$ and $H^{i}_{\fm}(M)=0$ for $i=0,1$ by Lemma~\ref{Lemma3}. Hence, the lemma follows in view of the identifications $[H^{i+1}_{\fm}(M)]_{n} \simeq H^{i}(X,\mathscr{M} \otimes \mathscr{O}_X(nD))$.
\end{proof}

\begin{remark}
\label{rmk4}
There is a characterization of ACM varieties in terms of the vanishing of cohomology groups without twisting sheaves. Let $X$ be a Cohen-Macaulay projective variety. If $H^i(X,\mathscr{O}_X)=0$ for all $0<i<\dim X$, then there exists a Cohen-Macaulay graded ring $R$ such that $X \simeq \Proj(R)$. For the proof of this fact, see \cite[(5.1.11)]{GoWa} and \cite[Proposition 6.1]{Smith}. More generally, if $E$ is an ample Cartier divisor on $X$, the ring $R(X,E)$ is CM if and only if $X$ is CM and $H^i(X,\mathscr{O}_X(nE))=0$ for all $0<i<\dim X$ and all $n \in \mathbb{N}$.
\end{remark}

The following proposition was originally proved in Grothendieck's EGA (see \cite[P. 579]{GorWed}) in the case that $\mathscr{M}=\mathscr{O}_X$.

\begin{proposition}
\label{Proposition2}
Let $f:X \to S$ be a projective surjective morphism of Noetherian schemes with equi-dimensional fibers. Let$\mathscr{M}$ be a coherent $\mathscr{O}_{X}$-module which is flat over $S$ such that $\Supp (\mathscr{M})=X$ and $\mathbf{P}=CM$. Then $U_{f}^{\mathscr{M}}(\mathbf{P})$ is Zariski open in $S$.
\end{proposition}

\begin{proof}
Every fiber of $f$ is an equi-dimensional projective scheme over a field by assumption. Moreover, since $\Supp (\mathscr{M})=X$, we have that $\Supp (\mathscr{M}_s)=X_s$ for every $s \in S$ by Nakayama's lemma. Now assume that $\mathscr{M}_s$ is a CM sheaf for some $s \in S$. Let $D$ be an $f$-ample Cartier divisor on $X$. Then $D_t$ is an ample divisor on the fiber $X_t$ for all $t \in S$. Since $\mathscr{M}_s$ is assumed to be CM, we have $H^i(X_{s},\mathscr{M}_s \otimes \mathscr{O}_{X_s}(-nD_s))=0$ for all $i <\dim X_{s}$ and $n \gg 0$ by \cite[Corollary 5.72]{KolMo}. We know that the fiber dimension $t \in S \mapsto \dim X_t$ is upper-semicontinuous by projectivity of $f$ and the function $t \in S \mapsto \dim_{k(t)} H^i(X_{t},\mathscr{M}_t \otimes \mathscr{O}_{X_t}(-nD_t))$ is upper-semicontinuous by projectivity of $f$ and flatness of $\mathscr{M}$ over $S$. Then there exists an open subset $s \in U \subseteq S$ such that $H^i(X_{t},\mathscr{M}_t \otimes \mathscr{O}_{X_t}(-nD_t))=0$ for all $t \in U$, $i<\dim X_t$ and $n \gg 0$. We compete the proof of the proposition by \cite[Corollary 5.72]{KolMo}.
\end{proof}

\begin{theorem}[Semicontinuity of ACM sheaves]
\label{Theorem2}
Let $f:X \to S$ be a projective surjective morphism of Noetherian schemes whose geometric fibers are connected and normal. Fix an $f$-ample divisor $D$. Let $\mathscr{M}$ be a coherent $\mathscr{O}_{X}$-module which is flat over $S$ such that $\Supp (\mathscr{M})=X$ and $\mathbf{P}=ACM$. Then $U_{f}^{\mathscr{M}}(\mathbf{P})$ with respect to $D$ in the fiberwise sense is Zariski open in $S$. 
\end{theorem}

\begin{proof}
Recall that the ACM condition is defined on a normal projective variety. By assumption, if $X_t$ denotes the fiber of $f$ for $t \in S$, then the base change $X_t \times \Spec \overline{k(t)}$ is a connected normal projective variety.

Assume that $\mathscr{M}_s$ is an ACM sheaf with respect to $D_s$ for some $s \in S$. Then $\mathscr{M}_s$ is CM. By applying Proposition \ref{Proposition2}, we may assume that $\mathscr{M}_t$ is CM for every $t \in S$ by shrinking $S$ to its smaller affine open neighborhood of $s \in S$ as above. Now by Lemma \ref{Lemma4}, we have $H^i(X_{s},\mathscr{M}_s \otimes \mathscr{O}_{X_s}(nD_s))=0$ for all $0< i <\dim X_{s}$ and $n \in \mathbb{Z}$. Using the upper-semicontinuity of the fiber dimension and the function: $
\dim_{k(t)} H^i(X_{t},\mathscr{M}_t \otimes \mathscr{O}_{X_t}(nD_t))~(t \in S)$, it follows that there exists an open subset $s \in U \subseteq S$ such that $H^i(X_{t},\mathscr{M}_t \otimes \mathscr{O}_{X_t}(nD_t))=0$ for all $t \in U$, $0< i <\dim X_t$, and $n \in \mathbb{Z}$, which is the desired conclusion by Lemma \ref{Lemma4}.
\end{proof}

The MCM condition on modules is closely related to $F$-purity of Noetherian rings of characteristic $p>0$. We refer the reader to \cite[Theorem 4.4]{ShiZha} and \cite{ShiZha2}. See also \cite[Theorem 5.8]{Hashi1}. We will discuss the global splitting of the Frobenius morphism on a projective variety in the next section.

\section{Characterization of globally $F$-regular varieties}

As was mentioned in the last part of the previous section, it is expected that ACM sheaves are closely related to the global $F$-regularity for projective varieties. As a clear evidence for this, we obtain a characterization of global $F$-regularity in terms of the Cohen-Macaulay property on the cokernel of a sheaf map defined by the Frobenius morphism. We also establish the semicontinuity property along fibers.

We begin with recalling the definition of globally $F$-regular varieties, a class of varieties introduced by Smith \cite{Smith}. Let $X$ be a projective variety over a field $k$ of characteristic $p > 0$ such that $[k:k^{p}] < \infty$. Every residue field of an algebra essentially of finite type over an $F$-finite field is again $F$-finite.

Let us make some review on the splitting of maps between modules. Let $f:M \to N$ be a map of modules over a Noetherian ring $R$. Then $f$ is \textit{pure}, if the induced map of $R$-modules: $f \otimes \mathrm{id}:M \otimes_{R} K \to N \otimes_{R} K$ stays injective for every $R$-module $K$. If the quotient $M/N$ is finitely generated, then the splitting of $f$ is equivalent to the purity of $f$. Here is another useful criterion for purity.

\begin{lemma}[Hochster-Huneke]
\label{Lemma5}
Assume that $(R,\fm)$ is either graded or Noetherian local ring and $E$ is an injective hull of the residue field of $R$. Let $f:M \to N$ be an $R$-module map, where $M$ is a finitely generated free $R$-module. Then $f:M \to N$ is pure if and only if the induced map of $R$-modules $f \otimes \mathrm{id}:M \otimes_{R} E \to N \otimes_{R} E$ is injective.
\end{lemma}

\begin{proof}
We refer the reader to \cite[Exercise 9.3]{Hu}. There, the proof is given only for local rings. However, as the injective hull exists in the category of graded rings, the lemma holds for a graded ring $(R,\fm)$ as well.
\end{proof}

Now let $R$ be a reduced Noetherian ring of characteristic $p > 0$ such that $R^{\frac{1}{p}}$ is module-finite over $R$ (such a ring is called \textit{F-finite}). Then $R$ is \textit{strongly F-regular} if for $c \in R$ not in every minimal prime of $R$, the $R$-module map $R \to R^{\frac{1}{q}}$ sending $1$ to $c^{\frac{1}{q}}$ splits for $q \gg 0$. For more results concerning this notion, we refer to \cite{HoHu1}. Here we only mention the following:

\begin{proposition}[Hochster-Huneke]
\label{Proposition3}
Assume that $R$ is a reduced $F$-finite Noetherian ring of characteristic $p>0$. Then the following hold:

\begin{enumerate}
\item
Any regular ring is strongly $F$-regular.

\item
$R$ is strongly $F$-regular if and only if $R_{P}$ is strongly $F$-regular for every prime $P$ of $R$.

\item
Let $c \in R$ be any element such that it is not in any minimal prime of $R$ and $R_{c}$ is strongly $F$-regular. Then $R$ is strongly $F$-regular if and only if the map of $R$-modules $R \to R^{\frac{1}{q}}; 1 \to c^{\frac{1}{q}}$ splits for some $q=p^e$.

\end{enumerate}
\end{proposition}

\begin{proof}
All statements are found in \cite[Theorem 5.5 and Theorem 5.9]{HoHu2}.
\end{proof}

In view of the above proposition, we say that a Noetherian $\mathbb{F}_p$-scheme $X$ is \textit{strongly F-regular}, if it admits an affine open covering, each of which affine piece is $F$-finite and strongly $F$-regular.

\begin{definition}[Smith]
A projective variety $X$ over an $F$-finite field is \textit{globally F-regular} if the section ring of $X$ is strongly $F$-regular for some ample divisor $D$ on $X$.
\end{definition}

Globally $F$-regular varieties are arithmetically normal and ACM. Being globally $F$-regular is really a global property. Even if every local ring of a variety is regular, it is usually far from globally $F$-regular. Some basic properties are found in \cite{Smith}.

\begin{example}
Take $R=\mathbb{F}_p[X,Y,Z]/(X^2+Y^3+Z^5), p \ge 7$ with the grading: $\deg X=15$, $\deg Y=10$, and $\deg Z=6$. Then $R$ is strongly $F$-regular and normal Gorenstein (see \cite{HoHu1} and \cite{Hu} for similar examples). Moreover, as the direct summand of a strongly $F$-regular domain is strongly $F$-regular, the Veronese subring $R^{(d)}$ is strongly $F$-regular. This fact reflects Lemma \ref{Lemma6} below. Globally $F$-regular varieties are a sub-class of Frobenius-split varieties (see below for the definition). Recently, some connections between globally $F$-regular varieties and varieties of Fano type have been clarified \cite{GonOkSanTak} and \cite{SchSmi}.
\end{example}

Let $X$ be an $\mathbb{F}_{p}$-scheme and let $F^{e}_{X}:X \to X$ denote the $e$-th \textit{absolute Frobenius morphism} with $e > 0$, which induces a morphism $\mathscr{O}_{X} \to (F^{e}_{X})_{*}\mathscr{O}_{X}$ of $\mathscr{O}_{X}$-modules. This map is the $p^e$-th power map at stalks, while it is an identity on the underlying topological spaces. Let $D$ be an effective Cartier divisor on $X$. The natural sheaf map $\mathscr{O}_{X} \to \mathscr{O}_{X}(D)$ descends to a morphism $(F^{e}_{X})_{*}\mathscr{O}_{X} \to (F^{e}_{X})_{*}\mathscr{O}_{X}(D)$. For the following lemma, see \cite[Theorem 3.10]{Smith} and \cite[Theorem 1]{Hashi2}.

\begin{lemma}
\label{Lemma6}
Suppose that $X$ is a projective variety over an $F$-finite field $k$ of characteristic $p>0$. Then the following are equivalent:

\begin{enumerate}
\item
$X$ is globally $F$-regular;

\item
the section ring of $X$ is strongly $F$-regular for any ample Cartier divisor $D$ on $X$;

\item
there exists an ample effective Cartier divisor $D$ such that $X-D$ is strongly $F$-regular, and the composed map:
$$
\mathscr{O}_{X} \to (F^{e}_{X})_{*}\mathscr{O}_{X} \to (F^{e}_{X})_{*}\mathscr{O}_{X}(D)
$$ 
admits a splitting for some $e > 0$.
\end{enumerate}
\end{lemma}

In particular, if the above composed sheaf map splits for some $e>0$, then it also splits for all $e' \ge e$. A projective variety $X$ is \textit{Frobenius-split} \cite{MeRam}, if $\mathscr{O}_{X} \to (F_{X})_*\mathscr{O}_{X}$ splits as a map of $\mathscr{O}_{X}$-modules. This is equivalent to requiring that $X$ admits a section ring which is $F$-pure for some ample divisor on $X$. Clearly, globally $F$-regular varieties are Frobenius-split. Moreover, if the section ring for some ample divisor is $F$-pure, then the section ring for any ample divisor is also $F$-pure \cite[Proposition 3.1]{Smith}.

Now let $f:X \to S$ be a morphism of $\mathbb{F}_{p}$-schemes. Then there is the following natural commutative diagram:
$$
\begin{CD}
X @>F^{e}_{X/S}>> X^{(e)} @>q>> X \\
@| @VVf^eV @VVfV \\
X @>f>> S @>F^{e}_{S}>> S \\
\end{CD}
$$
in which the right square is cartesian and we define the $e$-th \textit{relative Frobenius morphism} $F^{e}_{X/S}:X \to X^{(e)}$ for $X^{(e)}:=X \times_{S,F^e_S} S$.

\begin{lemma}
\label{Lemma7}
Let $f:X \to S$ be a morphism of $\mathbb{F}_p$-schemes. Then $F^{e}_{X/S}:X \to X^{(e)}$ is a bijection on the underlying topological spaces. Moreover, we have equality of stalks: $((F^{e}_{X/S})_*\mathscr{F})_{y}=\mathscr{F}_x$ for an $\mathscr{O}_{X}$-module $\mathscr{F}$ and $y=F^{e}_{X/S}(x)$.
\end{lemma}

\begin{proof}
This is an easy exercise.
\end{proof}

To discuss the global $F$-regularity along fibers of a morphism of schemes, we need a universal version of global $F$-regularity. By abuse of notation, we write $F^{e}_{X/k}$ for $F^{e}_{X/\Spec k}$. Let $q:X^{(e)} \to X$ be the map defined as above.

\begin{definition}
Let $f:X \to S=\Spec k$ be a projective variety, where $k$ is an $F$-finite field. Then $X$ is \textit{globally F-regular of type k}, if there exist an ample effective Cartier divisor $D$ such that for $e \gg 0$, $X^{(e)}-q^*D$ is strongly $F$-regular and the map of $\mathscr{O}_{X^{(e)}}$-modules: 
$$
\mathscr{O}_{X^{(e)}} \to (F^{e}_{X/k})_*\mathscr{O}_{X} \to (F^e_{X/k})_*\mathscr{O}_{X}(D)
$$ 
admits a splitting.
\end{definition}

\begin{lemma}
\label{Lemma8}
Let $X$ be a projective variety over an $F$-finite field $k$. Then $X$ is globally $F$-regular of type $k$ if and only if the variety $X_{L}:=X \times_{\Spec k} \Spec L$ is globally $F$-regular for any finite field extension $k \to L$.
\end{lemma}

\begin{proof}
Since $q:X^{(e)} \to X$ is a finite flat morphism, the sheaf map of $\mathscr{O}_X$-modules $\mathscr{O}_X \to q_*\mathscr{O}_{X^{(e)}}$ splits. Suppose first that $X$ is globally $F$-regular of type $k$ and let $k \to L$ be a finite field extension. Let $k^{\sep}$ be the separable closure of $k$ in $L$. Then $X_{k^{\sep}} \to X$ is finite \'etale. Using this fact, it is checked that $X_{k^{\sep}}$ is globally $F$-regular of type $k^{\sep}$. By field theory, $k^{\sep} \to L$ is purely inseparable. So we are reduced to the case that $k \to L$ is purely inseparable. We know that
\begin{equation}
\label{splitmap}
\mathscr{O}_{X^{(e)}} \to (F^e_{X/k})_*\mathscr{O}_{X} \to (F^e_{X/k})_*\mathscr{O}_{X}(D)
\end{equation}
splits for $e \gg 0$. Take sufficiently large $e>0$ such that $L$ is contained in $k^{\frac{1}{p^e}}$. Then we have a natural morphism $q_L:X^{(e)} \to X_L$. Pushing $\mathscr{O}_X \to q_*\mathscr{O}_{X^{(e)}}$ down via $F^e_{X/k}$, we get a map $(F^e_{X/k})_*\mathscr{O}_X \to (F^e_{X^{(e)}})_*\mathscr{O}_{X^{(e)}}$ due to $F^e_{X^{(e)}}=F^e_{X/k} \circ q$. Since this map splits, there is a map $\sigma: (F^e_{X^{(e)}})_*\mathscr{O}_{X^{(e)}} \to (F^e_{X/k})_*\mathscr{O}_X$ that splits it. Next, tensoring $\mathscr{O}_X \to q_*\mathscr{O}_{X^{(e)}}$ with $\mathscr{O}_X(D)$ and pushing it down via $F_{X/k}^e$, we get a map $(F^e_{X/k})_*\mathscr{O}_X(D) \to(F^e_{X^{(e)}})_* \mathscr{O}_{X^{(e)}}(q^*D) $ by the projection formula. There is a map $\sigma': (F^e_{X^{(e)}})_* \mathscr{O}_{X^{(e)}}(q^*D) \to  (F^e_{X/k})_*\mathscr{O}_X(D)$ that splits it and compatible with $\sigma$. Now in view of $(\ref{splitmap})$, we have a commutative diagram: 
$$
\begin{CD}
\mathscr{O}_{X^{(e)}} @>>> (F^e_{X^{(e)}})_* \mathscr{O}_{X^{(e)}} @>>> (F^e_{X^{(e)}})_* \mathscr{O}_{X^{(e)}}(q^*D) \\
@| @V\sigma VV @V\sigma'VV \\
\mathscr{O}_{X^{(e)}} @>>> (F^e_{X/k})_*\mathscr{O}_X @>>> (F^e_{X/k})_*\mathscr{O}_X(D)
\end{CD}
$$
Note that $X^{(e)}-q^*D$ is strongly $F$-regular for $e \gg 0$ by assumption. Since the bottom horizontal map splits by our assumption, so does the upper horizontal map. We conclude that $X^{(e)}$ is globally $F$-regular. Let $q_L:X^{(e)} \to X_L$ be as above. Then since $(q_L)_*\mathscr{O}_{X^{(e)}}$ is a free $\mathscr{O}_{X_L}$-module, it follows that $X_L$ is globally $F$-regular.

We prove the converse. Then $X^{(e)}$ is globally $F$-regular for any $e>0$, because $X^{(e)}=X \times \Spec k^{\frac{1}{p^e}}$. We may choose an ample effective divisor $D$ on $X$ such that $X^{(e)}-q^*D$ is strongly $F$-regular and 
\begin{equation}
\label{splitmap2}
\mathscr{O}_{X^{(e)}} \to (F^e_{X^{(e)}})_* \mathscr{O}_{X^{(e)}} \to (F^e_{X^{(e)}})_* \mathscr{O}_{X^{(e)}}(q^*D)
\end{equation}
splits for $e \gg 0$. The requisite maps for constructing the diagram have been given in the first step. Together with $(\ref{splitmap2})$, we have a commutative diagram:
$$
\begin{CD}
\mathscr{O}_{X^{(e)}} @>>> (F^e_{X^{(e)}})_* \mathscr{O}_{X^{(e)}} @>>> (F^e_{X^{(e)}})_* \mathscr{O}_{X^{(e)}}(q^*D) \\
@| @AAA @AAA \\
\mathscr{O}_{X^{(e)}} @>>> (F^e_{X/k})_*\mathscr{O}_X @>>> (F^e_{X/k})_*\mathscr{O}_X(D) \\
\end{CD}
$$
Since the top horizontal map splits by our assumption, so does the bottom horizontal map.
\end{proof}

In view of this lemma, globally $F$-regular of type $k$ is the same as the geometrically globally $F$-regular over $k$. We prove the main theorem. First, we need some discussions on the ubiquity on Gorenstein rings, which is due to Prof. K-i. Watanabe. Since he has never published this idea, we explain it in detail. The author is grateful to him.

\begin{proposition}[K-i. Watanabe]
\label{newdiscussion}
Let $X$ be a connected normal projective variety over an algebraically closed field $k$ with $\dim X \ge 2$. Suppose that $X$ is Cohen-Macaulay and satisfies $H^i(X,\mathscr{O}_X)=0$ for all $0<i<\dim X$. Then there exists an ample $\mathbb{Q}$-divisor $E$ such that the generalized section ring $R:=R(X,E)$ is Gorenstein.
\end{proposition}

\begin{proof}
First we fix notation. Let $E \in \Div(X)_{\mathbb{Q}}$. Write
$$
E=F+\sum_V \frac{p_V}{q_V} V,
$$ 
where we assume $F \in \Div(X)$, $(p_V,q_V)=1$ and $q_V>p_V \ge 1$. Let $E':=\sum_V \frac{q_V-1}{q_V} V$.

Choose an ample divisor $H$ such that the divisor $H-2K_X$ is very ample. Let $X \to \mathbb{P}^n_k$ be an embedding defined by the linear system $|H-2K_X|$. Then by Bertini's theorem (this is where we need $\dim X \ge 2$), we may find an irreducible subvariety $V \in \Div(X)$ such that $H-2K_X \sim V$ and let $E:=K_X+\frac{1}{2} V$, where $F=K_X$ and $E'=\frac{1}{2}V$. Moreover, we have $H \sim 2E$. By Serre's vanishing theorem and \cite[Corollary 5.72]{KolMo}, taking $H$ as above, sufficiently ample, we have $H^i(X,\mathscr{O}_X(nE))=0$ for all $0<i < \dim X$ and $|n| \gg 0$. Replacing $H$ again with its high power, the assumption that $H^i(X,\mathscr{O}_X)=0$ for $0<i < \dim X$ shows that $H^i(X,\mathscr{O}_X(nE))=0$ for all $0<i < \dim X$ and $n \in \mathbb{Z}$. Then $R:=R(X,E)$ is Cohen-Macaulay in view of Remark \ref{rmk4}. Moreover, since $K_X+E'+nE=(n+1)E$, we have $K_R \simeq R(1)$ by \cite[ Corollary 2.9]{Wa}. So $R$ is Gorenstein.
\end{proof}

\begin{theorem}
\label{Theorem3}
Let $X$ be a connected smooth projective variety over an algebraically closed field of characteristic $p>0$ with $\dim X \ge2$. Fix an ample effective divisor $D$ and let
\begin{equation}
\label{exact}
\mathscr{O}_X \to (F^{e}_X)_*\mathscr{O}_{X} \to (F^{e}_X)_*\mathscr{O}_{X}(D)
\end{equation}
be the composed map of $\mathscr{O}_X$-modules, together with its cokernel sheaf $\mathscr{H}_e$. Assume further that $H^{i}(X,\mathscr{O}_X)=0$ for all $0 < i < \dim X$. Consider the following conditions:
 
\begin{enumerate}
\item
$\mathscr{H}_e$ is an ACM sheaf with respect to any ample $\mathbb{Q}$-divisor for $e \gg 0$.

\item
The sequence $(\ref{exact})$ splits; that is, $X$ is globally $F$-regular.

\item
$\mathscr{H}_e$ is an ACM sheaf with respect to any ample Cartier divisor for $e \gg 0$.
\end{enumerate}
Then we have implications $(1) \Rightarrow (2) \Rightarrow (3)$.
\end{theorem}

\begin{proof}
We first establish $(2) \Rightarrow (3)$. Then from $(\ref{exact})$, the following exact sequence of $\mathscr{O}_{X}$-modules:
\begin{equation}
\label{exact1}
0 \to \mathscr{O}_X \to (F^{e}_X)_*\mathscr{O}_{X}(D) \to \mathscr{H}_e \to 0
\end{equation}
splits. Choose an ample Cartier divisor $E$. We want to show that $\mathscr{H}_e$ is ACM with respect to $E$ for $e \gg 0$. We get a split short exact sequence:
\begin{equation}
\label{exact2}
0 \to \mathscr{O}_{X}(nE) \to (F^{e}_{X})_*\mathscr{O}_{X}(D) \otimes \mathscr{O}_X(nE) \to \mathscr{H}_e\otimes\mathscr{O}_X(nE) \to 0 
\end{equation}
and the sheaves in the sequence $(\ref{exact2})$ are coherent. The sequence $(\ref{exact2})$ induces a long exact sequence:
$$
\cdots \to H^i(X,\mathscr{O}_{X}(nE)) \to H^i(X,(F^{e}_{X})_{*}\mathscr{O}_{X}(D) \otimes \mathscr{O}_X(nE)) 
$$
$$
\to H^i(X,\mathscr{H}_e\otimes \mathscr{O}_X(nE)) \to  H^{i+1}(X,\mathscr{O}_{X}(nE)) \to \cdots.
$$ 
By the splitting hypothesis, the map $H^i(X,(F^{e}_{X})_*\mathscr{O}_{X}(D) \otimes \mathscr{O}_X(nE)) 
\to H^i(X,\mathscr{H}_e \otimes \mathscr{O}_X(nE))$ is surjective. Now we show that
\begin{equation}
\label{exact3}
H^{i}(X,(F^{e}_{X})_*\mathscr{O}_{X}(D) \otimes \mathscr{O}_X(nE))=0
\end{equation}
for all $0<i<\dim X$, $n \in \mathbb{Z}$ and $e \gg 0$. The pull-back under $F^e_X$ gives $(F^e_X)^*\mathscr{O}_X(nE)=\mathscr{O}_X(p^enE)$ and we have
$$
H^i(X,\mathscr{O}_{X}(D+p^enE)) \simeq H^i(X,(F^{e}_{X})_*\mathscr{O}_{X}(D+p^enE)) \simeq  H^i(X,(F^{e}_{X})_*\mathscr{O}_{X}(D) \otimes \mathscr{O}_X(nE)) 
$$ 
by \cite[Chapter III, Exercise 8.2]{Har}, together with the projection formula. Now we claim that
$$
H^i(X,\mathscr{O}_{X}(D+p^enE))=0
$$ 
for all $0 < i <\dim X$, $n \in \mathbb{Z}$ and $e \gg 0$. Indeed if $n > 0$, then Serre's vanishing theorem gives the result. If $n=0$, this follows from \cite[Theorem 4.10]{Smith}. Finally if $n<0$, then since $\mathscr{O}_X(-p^enE)$ is ample, the divisor $\mathscr{O}_X(-D-p^enE)$ is also ample for $e \gg 0$. So applying \cite[Corollary 4.4]{Smith}, we have $H^i(X,\mathscr{O}_X(D+p^enE))=0$ for all $i<\dim X$, which yields the vanishing $(\ref{exact3})$. Therefore, $H^i(X,\mathscr{H}_e \otimes \mathscr{O}_X(nE))=0$ for all $0 < i < \dim X$, $n \in \mathbb{Z}$ and $e \gg 0$. On the other hand, localizing the split exact sequence $(\ref{exact1})$ at all points of $\Supp \mathscr{H}_e$ and applying Lemma \ref{Lemma7}, the splitting of the sequence implies that $\mathscr{H}_e$ is a CM sheaf. Hence $\mathscr{H}_e$ is ACM with respect to $E$ in view of Lemma \ref{Lemma4}.

We next establish $(1) \Rightarrow (2)$. It suffices to show that the sequence $(\ref{exact})$ splits. We will use the local duality for graded rings. By Proposition \ref{newdiscussion} together with the assumption that $H^i(X,\mathscr{O}_X)=0$ for all $0 < i < \dim X$ and $\dim X\ge2$, there exists an ample $\mathbb{Q}$-divisor $E \in \Div(X)_{\mathbb{Q}}$ such that the generalized section ring
$$
R:=R(X,E)=\bigoplus_{n \in \mathbb{Z}} H^{0}(X,\mathscr{O}_{X}(nE))
$$
is Gorenstein and $X \simeq \Proj(R)$. Let us put
$$
M:=\bigoplus_{n \in \mathbb{Z}} H^{0}(X,(F^{e}_{X})_{*}\mathscr{O}_{X}(D) \otimes \mathscr{O}_{X}(nE)).
$$ 
Then to show that $\mathscr{O}_{X} \to (F^{e}_{X})_{*}\mathscr{O}_{X}(D)$ splits, it suffices to show that the injective map of graded $R$-modules: $R \to M$ splits. Note that the sequence $(\ref{exact2})$ stays exact, since $\mathscr{O}_X(nE)$ is an invertible sheaf.

We have $\bigoplus_{n \in \mathbb{Z}} H^{1}(X,\mathscr{O}_{X}(nE))=0$ as seen in the proof of Proposition \ref{newdiscussion}. Thus the cokernel of the map $R \to M$ is isomorphic to $N:=\bigoplus_{n \in \mathbb{Z}} H^{0}(X,\mathscr{H}_e \otimes \mathscr{O}_X(nE))$. Then we have a short exact sequence of graded $R$-modules:
\begin{equation}
\label{exact4}
0 \to R \to M \to N \to 0
\end{equation}
and $N$ is a graded MCM module over $R$ by our hypothesis. According to Lemma \ref{Lemma5}, the map $(\ref{exact4})$ splits if and only if the induced map $E_R \to M \otimes_R E_R$ is injective, where $E_R$ is the injective hull of the residue field of $R$.

By construction, $R$ is Gorenstein and $K_R$ is free over $R$. Then the induced sequence
$$
0 \to K_R \to M \otimes_R K_R \to N \otimes_R K_R \to 0
$$
is short exact. Let $d=\dim R$ and let $\fm$ be the graded maximal ideal of $R$. We have a long exact sequence:
$$
\cdots \to H^{d-1}_{\fm}(N \otimes_R K_R) \to H^{d}_{\fm}(K_R) \to H^d_{\fm}(M \otimes_R K_R) \to \cdots
$$
Then it follows that $H^{d-1}_{\fm}(N \otimes_R K_R)=0$ by the depth criterion for MCM modules, together with the fact that $K_R$ is a free $R$-module.

Now we can compute the injective map $H^d_{\fm}(K_R) \to H^d_{\fm}(M \otimes_R K_R)$. Since these are top local cohomology modules, we see that it is equal to $E_R \to M \otimes_R E_R$, due to the fact $H^d_{\fm}(K_R)\simeq E_R$ (the isomorphism $M \otimes_R H^d_{\fm}(K_R) \simeq H^{d}_{\fm}(M \otimes_R K_R)$ follows from the direct computation for top local cohomology). Hence it is injective. So the sequence $(\ref{exact4})$ splits and $X$ is globally $F$-regular, as desired.
\end{proof}

The author does not know if $(3)$ implies $(1)$. The reason for taking $\mathbb{Q}$-divisors is that it is necessary to find a Gorenstein ring as a generalized section ring. The following corollary gives a criterion for global $F$-regularity in terms of the generalized section ring.

\begin{corollary}[Criterion for global $F$-regularity]
\label{Finalcorollary}
Let $X$ be a connected normal projective variety over an $F$-finite field with $\dim X \ge 2$. Suppose that the following condition holds:
\begin{enumerate}
\item[$\bullet$]
There exists a generalized section ring $R=R(X,E)$ that is Gorenstein such that there is an injective $R$-module map $R \to R^{\frac{1}{q}};1 \mapsto c^{\frac{1}{q}}$, $R_c$ is strongly $F$-regular, where $c\in R$ is in no minimal prime of $R$, and its cokernel is a MCM module for some $q=p^e$.
\end{enumerate}
Then $X$ is globally $F$-regular.
\end{corollary}

\begin{proof}
We need to show that the section ring of $X$ with respect to an ample divisor is strongly $F$-regular. Let $N$ be the cokernel module of $R \to R^{\frac{1}{q}}$. Since $R$ is Gorenstein, the canonical module $K_R$ is free. Then applying the local cohomology functor to the short exact sequence
$$
0 \to K_R \to R^{\frac{1}{q}} \otimes_R K_R \to N \otimes_R K_R \to 0,
$$
we see that $R \to R^{\frac{1}{q}}$ splits by using the fact the $R$-module $N$ is MCM and thus, $R=R(X,E)$ is strongly $F$-regular. Let $m>0$ be such that $mE$ is an ample Cartier divisor. By taking the $m$-th Veronese subring of $R$, it follows that the section ring $R^{(m)}:=R(X,mE)$ is strongly $F$-regular. Hence $X$ is globally $F$-regular.
\end{proof}

\begin{remark}
It is interesting to try to extend our results to the pair $(X,D)$, which is important for applications. Some relevant ideas may be found in \cite{SchSmi}. It is also interesting to consider the splitting problem of $\mathscr{O}_{X^{(e)}} \to (F^e_{X/S})_*\mathscr{O}_X \to (F^e_{X/S})_*\mathscr{O}_X(D)$. The splitting problem of $\mathscr{O}_{X^{(e)}} \to (F^{e}_{X/S})_*\mathscr{O}_{X}$ is addressed in \cite{Hashi3} when $X$ and $S$ are affine.
\end{remark}

Let $f:X \to S$ be a morphism of $\mathbb{F}_p$-schemes with $s \in S$. Write $\mathscr{M}_s$ (resp. $D_s$) for the restriction $\mathscr{M}|_{X_s}$ (resp. $D|_{X_s}$). Let $(X_s)^{(e)}=X_s \times_{\Spec k(s)} \Spec k(s)^{\frac{1}{p^e}}$ and let $(X^{(e)})_s$ denote the fiber of $f^e:X^{(e)} \to S$ over $s$. Then we have $(X_s)^{(e)} \simeq (X^{(e)})_s$ canonically, which we simply write as $X^{(e)}_s$. Finally, we prove the following corollary. A similar result for Frobenius-split varieties is also known \cite{KiRa}.

\begin{corollary}[Semicontinuity of global $F$-regularity]
Let $f:X \to S$ be a projective smooth morphism of schemes of finite type over an $F$-finite field and let 
$$
\mathbf{P}=\mbox{globally F-regular of type}~k(s)
$$
for $s \in S$. Then $U_{f}(\mathbf{P})$ is Zariski open in $S$. 
\end{corollary}

\begin{proof}
Assume that $X_{s}$ is globally $F$-regular of type $k(s)$ for some point $s \in S$ (the residue field $k(s)$ is $F$-finite). For an effective $f$-ample Cartier divisor $D$, consider the composed map of $\mathscr{O}_{X^{(e)}}$-modules:
$$
\mathscr{O}_{X^{(e)}} \to (F^{e}_{X/S})_{*}\mathscr{O}_{X} \to (F^{e}_{X/S})_{*}\mathscr{O}_{X}(D).
$$
After shrinking $S$ to a smaller open neighborhood, we may assume that all the fibers of $f:X \to S$ are equi-dimensional and the dimension of every fiber is equal to $d-1$. Grothendieck's theory of duality tells us that the relative dualizing complex $f^{!}\mathscr{O}_S$ has only one nonzero homology and write it as $\omega_{f}$, which is locally free due to the smoothness of $f$. Since $f^e:X^{(e)} \to S$ is smooth, $\omega_{f^e}$ is also locally free.

Let us interpret everything in terms of sheaf cohomology. Note that $p_{y}:X^{(e)}_y \to X_y$ is finite flat and the sheaf $p^{*}_{y}\mathscr{O}_{X_y}(nH_y) \simeq \mathscr{O}_{X^{(e)}_y}(nH_y^{(e)})$ is ample for any $y \in S$ and an arbitrary fixed $f$-ample effective divisor $H$ on $X$. Let us consider the map:
$$
R:=\bigoplus_{n \in \mathbb{Z}}H^0\big(X^{(e)}_s,\mathscr{O}_{X^{(e)}_s}(nH_s^{(e)})\big) \to M:=\bigoplus_{n \in \mathbb{Z}}H^0\big(X^{(e)}_s,(F^e_{X_s/k(s)})_{*}\mathscr{O}_{X_s}(D_s)\otimes \mathscr{O}_{X^{(e)}_s}(nH_s^{(e)})\big).
$$
Since $X_s$ is globally $F$-regular of type $k(s)$ by assumption, the $R$-module map $R \to M$ splits. Then as discussed in the proof of Theorem \ref{Theorem3}, we find that the following natural map:
\begin{equation}
\label{injective}
H^{d}_{\fm}(K_{R}) \simeq \bigoplus_{n \in \mathbb{Z}} H^{d-1}\big(X^{(e)}_s,\omega_{X^{(e)}_s} \otimes p^*_s\mathscr{O}_{X_s}(nH_s)\big) 
\end{equation}
$$
\to H^{d}_{\fm}(M \otimes_R K_{R}) \simeq \bigoplus_{n \in \mathbb{Z}} H^{d-1}\big(X^{(e)}_s,(F^{e}_{X_s/k(s)})_{*}\mathscr{O}_{X_s}(D_s) \otimes \omega_{X^{(e)}_s} \otimes p^*_s\mathscr{O}_{X_s}(nH_s)\big)
$$
is injective for the dualizing sheaf $\omega_{X^{(e)}_s}$ of $X^{(e)}_s$, where we use $\widetilde{K_R} \simeq \omega_{X^{(e)}_s}$. Since the fiber dimension of both $f$ and $f^e$ is constant and equal to $d-1$, it follows from Grothendieck's vanishing theorem \cite[Chapter III, Theorem 2.7]{Har} that
\begin{equation}
\label{fiber1}
H^d\big(X^{(e)}_y,\omega_{X^{(e)}_y} \otimes p^*_y\mathscr{O}_{X_y}(nH_y)\big)=0
\end{equation}
and
\begin{equation}
\label{fiber2}
H^d\big(X^{(e)}_y,(F^{e}_{X_y/k(y)})_{*}\mathscr{O}_{X_y}(D_y) \otimes \omega_{X^{(e)}_y} \otimes p^*_y\mathscr{O}_{X_y}(nH_y)\big)=0
\end{equation}
for all $y \in S$. 

Let $p:X^{(e)} \to X$ be the natural projection. Then $\omega_{f^e} \otimes p^*\mathscr{O}_{X}(nH)$ is coherent and flat over $S$. By $(\ref{fiber1})$, we can apply \cite[Corollary 2 (i), P. 50]{Mum} to yield an isomorphism:
\begin{equation}
\label{fiber3}
\mathcal{R}^{d-1}f^e_{*}\big(\omega_{f^e} \otimes p^*\mathscr{O}_{X}(nH)\big) \otimes k(y) \simeq H^{d-1}\big(X^{(e)}_y,\omega_{X^{(e)}_y} \otimes p^*_y\mathscr{O}_{X_y}(nH_y)\big).
\end{equation}
On the other hand, $(F^{e}_{X/S})$ is a flat morphism by the smoothness of $f$, so the sheaf $(F^{e}_{X/S})_{*}\mathscr{O}_{X}(D) \otimes \omega_{f^e} \otimes p^*\mathscr{O}_{X}(nH)$ is coherent and flat over $S$. Then by $(\ref{fiber2})$, we can apply \cite[Corollary 2 (i), P. 50]{Mum} to yield an isomorphism:
\begin{equation}
\label{fiber4}
\mathcal{R}^{d-1}f^e_{*}\big((F^{e}_{X/S})_{*}\mathscr{O}_{X}(D) \otimes \omega_{f^e} \otimes p^*\mathscr{O}_{X}(nH)\big)\otimes k(y)
\end{equation}
$$
\simeq H^{d-1}\big(X^{(e)}_y,(F^{e}_{X_y/k(y)})_{*}\mathscr{O}_{X_y}(D_y) \otimes \omega_{X^{(e)}_y} \otimes p^*_y\mathscr{O}_{X_y}(nH_y)\big)
$$
for $y \in S$. We note that $p^*\mathscr{O}_{X}(nH) \simeq \mathscr{O}_{X^{(e)}}(nH^{(e)})$ is $f^e$-ample. Let 
$$
\bigoplus_{n \in \mathbb{Z}} \mathcal{R}^{d-1}f^e_{*}\big(\omega_{f^e} \otimes p^*\mathscr{O}_{X}(nH)\big) \xrightarrow{\Phi} \bigoplus_{n \in \mathbb{Z}} \mathcal{R}^{d-1}f^e_{*}\big((F^{e}_{X/S})_{*}\mathscr{O}_{X}(D) \otimes \omega_{f^e} \otimes p^*\mathscr{O}_{X}(nH)\big)
$$
be the natural sheaf map. Then by $(\ref{fiber3})$ and $(\ref{fiber4})$, since the map $(\ref{injective})$ is injective, the induced map $\Phi \otimes k(s)$ is injective. This means that there exists an open neighborhood $s \in U \subseteq S$ such that $\Phi \otimes k(y)$ is injective for $y \in U$ by Nakayama's lemma. Again by $(\ref{fiber3})$ and $(\ref{fiber4})$, we find that
$$
\bigoplus_{n \in \mathbb{Z}} H^{d-1}\big(X^{(e)}_y,\omega_{X^{(e)}_y} \otimes p^*_y\mathscr{O}_{X_y}(nH_y)\big) 
$$
$$
\to \bigoplus_{n \in \mathbb{Z}} H^{d-1}\big(X^{(e)}_y,(F^{e}_{X_y/k(y)})_{*}\mathscr{O}_{X_y}(D_y) \otimes \omega_{X^{(e)}_y} \otimes p^*_y\mathscr{O}_{X_y}(nH_y)\big)
$$
is injective. Thus, the fiber $X_y$ is globally $F$-regular of type $k(y)$ by Theorem \ref{Theorem3}. This completes the proof of the corollary.
\end{proof}

\begin{acknowledgement}
The author thanks the anonymous referee for his or her meticulous reading of the manuscript and pointing out errors.
\end{acknowledgement}

\end{document}